\documentclass{amsart}
\usepackage{amsfonts}
\usepackage{amsmath}
\usepackage{amssymb}
\usepackage{amscd}
\usepackage{MnSymbol}

\usepackage[all,cmtip]{xy}

\newtheorem{thm}{Theorem}[section]
\newtheorem{lm}[thm]{Lemma}
\newtheorem{cor}[thm]{Corollary}
\newtheorem{prop}[thm]{Proposition}

\theoremstyle{remark}

\theoremstyle{remark}
\newtheorem{exm}[thm]{Example}

\DeclareMathOperator{\Ker}{\rm Ker}
\DeclareMathOperator{\Coker}{\rm cok}
\DeclareMathOperator{\Img}{\rm Im}

\newcommand{\N}{\mathcal{N}}
\newcommand{\J}{\mathcal{J}}
\newcommand{\Ac}{\mathcal{A}}
\newcommand{\ot}{\overline{t}}
\newcommand{\Bc}{\mathcal{B}}
\newcommand{\Ec}{\mathcal{E}}

\newcommand{\Nc}{\mathcal{N}}

\newcommand{\Uc}{\mathcal{U}}
\newcommand{\Fc}{\mathcal{F}}
\newcommand{\Pc}{\mathcal{P}}
\newcommand{\Ic}{\mathcal{I}}
\newcommand{\Gc}{\mathcal{G}}
\newcommand{\Vc}{\mathcal{V}}

\newcommand{\Cc}{\mathcal{C}}

\newcommand{\M}{\mathcal{M}}

\newcommand{\Sc}{\mathcal{S}}
\newcommand{\Tc}{\mathcal{T}}
\newcommand{\K}{\mathcal{K}}
\newcommand{\wK}{\mathcal{P}}

\newcommand{\onu}{\overline{\nu}}
\newcommand{\ou}{\overline{u}}
\newcommand{\omu}{\overline{\mu}}
\newcommand{\wnu}{\widetilde{\nu}}
\newcommand{\opi}{\overline{\pi}}
\newcommand{\orho}{\overline{\rho}}

\newcommand{\Hom}{\mbox{Hom}}
\newcommand{\Mod}{\mathbf{Mod}}

\newcommand{\im}{\mbox{Im}}


\begin{document}
\title{Compactness in abelian categories}
\author{Peter K\'alnai}
\address{Department of Algebra, Charles University in Prague,
Faculty of Mathematics and Physics Sokolovsk\' a 83, 186 75 Praha 8, Czech Republic} 
\email{kalnai@karlin.mff.cuni.cz}

\author{Jan \v Zemli\v cka}
\email{zemlicka@karlin.mff.cuni.cz}

\begin{abstract}
We relativize the notion of a compact object in an abelian category with respect to a fixed subclass of objects.
We show that the standard closure properties persist to hold in this case. Furthermore, we describe categorical and set-theoretical conditions under which all products of compact objects remain compact.
\end{abstract}
\date{\today}
\maketitle

An object $C$ of an abelian category $\Ac$ closed under coproducts is said to be {\it compact} if the covariant functor $\Ac(C,-)$ commutes with all direct sums, i.e. there is a canonical isomorphism between $\Ac(C,\bigoplus \mathcal D)$ and $\bigoplus \Ac(C,\mathcal D)$ in the category of abelian groups for every system of objects $\mathcal D$.
The foundations for a systematic study of compact objects in the context of module categories were laid in 60's by Hyman Bass \cite[p.54]{Bass68} in 60's. The introductory work on theory of dually slender modules goes back to Rudolf Rentschler \cite{Ren69} and further research of compact objects has been motivated by progress in various branches of algebra such as theory of representable equivalences of module categories \cite{ColMen93,ColTrl94}, the structure theory of graded rings \cite{GomMilNas94}, and almost free modules \cite{Tr95}.

From the categorically dual point of view discussed in \cite{ElbKep98}, commutativity of the contravariant functor on full module categories behaves a little bit differently. The equivalent characterizations of compactness split in this dual case into a hierarchy of strict implications dependent on the cardinality of commuting families. The strongest hypothesis assumes arbitrary cardinalities and it leads to the class of so called \emph{slim} modules (also known as \emph{strongly slender}), which is a subclass of the most general class of $\aleph_1$-slim modules (also called as \emph{slender}), which involves only commutativity with countable families. The authors proved that the cardinality of a non-zero slim module is greater than or equal to any measurable cardinal (and presence of such cardinality is also sufficient condition for existence of a non-zero slim module) and that the class of slim modules is closed under direct sums. Thus, absence of a measurable cardinal ensures that there is at least one non-zero slim module and in fact, abundance of them. On the other hand, if there is a proper class of measurable cardinals then there is no such object like a non-zero slim module. This motivated the question in the dual setting, namely if the class of compact objects in full module categories (termed also as \emph{dually slender} modules) is closed under direct products. Offering no surprise, set-theoretical assumptions have helped to establish the conclusion also in this case. 

The main objective of the present paper is to refine several results on compactness. The obtained improvement comes from transferring behavior of modules to the context of general abelian categories. In particular we provide a generalized description of classes of compact objects closed under product that was initially exposed for dually slender modules in \cite{KalZem2014}. Our main result Theorem~\ref{CompGenCatsClosedUnderProds} shows that the class of $\Cc$-compact objects of a reasonably generated category is closed under products whenever there is no strongly inaccessible cardinal. Note that this outcome is essentially based on the characterization of non-$\Cc$-compactness formulated in Theorem~\ref{e.c.}. Dually slender and self-small modules (which may be identically translated as self-dually slender) form naturally available instances of compact and self-compact objects (see e.g. \cite{EklGooTrl97} and \cite{Dvo2015}).

From now on, we suppose that $\Ac$ is an abelian category closed under arbitrary coproducts and products. We shall use the terms \emph{family} or \emph{system} for any discrete diagram, which can be formally described as a mapping from a set of indices to a set of objects. For unexplained terminology we refer to \cite{Goo79,Pop73}.

\section{Compact objects in abelian categories}

Let us recall basic categorical notions. A category with a zero object is called \emph{additive} if for every finite system of objects there exist the product and coproduct which are canonically isomorphic, every $\Hom$-set has a structure of abelian groups and the composition of morphisms is bilinear. An additive category is \emph{abelian} if there exists kernel and a cokernel for each morphism, monomorphisms are exactly kernels of some morphisms and epimorphisms cokernels. A category is said to be \emph{complete (cocomplete)} whenever it has all limits (colimits) of small diagrams.
Finally, a cocomplete abelian category where all filtered colimits of exact sequences preserve exactness is \emph{Ab5}. For further details on abelian category see e.g. \cite{Pop73}.

Assume $\M$ is a family of objects in $\Ac$. Throughout the paper, the corresponding coproduct is designated $(\bigoplus\M,(\nu_M \mid M \in \M))$ and the product $(\prod\M,(\pi_M \mid M\in\M))$. We call $\nu_M$ and $\pi_M$ as the \emph{structural morphisms} of the coproduct and the product, respectively. 

Suppose that $\Nc$ is a subfamily of $\M$. We call the coproduct $(\bigoplus\Nc,(\onu_N \mid N\in\Nc))$ in $\Ac$ as the \emph{subcoproduct} and dually the product $(\prod\Nc,(\opi_N \mid N\in\Nc))$ as the \emph{subproduct}. Note that there exist the unique canonical morphisms $\nu_\Nc \in \Ac \left( \bigoplus\Nc, \bigoplus\M \right)$ and $\pi_\Nc \in \Ac \left(\prod\M, \prod\Nc\right)$ given by the universal property of the colimit $\bigoplus\Nc$ and the limit $\prod\Nc$ satisfying $\nu_N=\nu_\Nc \circ \onu_N$ and $\pi_N=\opi_N \circ \pi_\Nc$ for each $N\in\Nc$, to which we refer as the \emph{structural morphisms} of the subcoproduct and the subproduct over a subfamily $\Nc$ of $\M$, respectively. The symbol $1_M$ is used for the identity morphism of an object $M$.

We start with formulation of two introductory lemmas which collects several basic but important properties of
the category $\Ac$ expressing relations between the coproduct and product over a family using their structural morphisms.  
 
\begin{lm}\label{AssMorphAbelianCats} 
Let $\Ac$ be a complete abelian category, $\M$ a family of objects of $\Ac$ with all coproducts and $\Nc\subseteq\M$. Then
\begin{enumerate}
\renewcommand{\labelenumi}{(\roman{enumi})}  
\item There exist unique morphisms $\rho_\Nc \in \Ac(\bigoplus\M, \bigoplus\Nc)$ and $\mu_\Nc \in \Ac(\prod\Nc, \prod\M)$ such that 
$\rho_\Nc \circ \nu_M=\onu_M$, $\pi_M \circ \mu_\Nc=\onu_M$  if $M\in\Nc$ and $\rho_\Nc \circ \nu_M=0$, $\pi_M \circ \mu_\Nc = 0$  if $M\notin\Nc$.  
\item For each $M\in \M$ there exist unique morphisms 
$\rho_M \in \Ac(\bigoplus\M, M)$ and  $\mu_M \in \Ac(M, \prod\M)$ such that $\rho_M \circ \nu_M=1_{M}$, $\pi_M \circ \mu_M=1_M$ and $\rho_M \circ \nu_N=0$, $\pi_N \circ \mu_M=0$ whenever $N\ne M$. If $\orho_M$ and $\omu_M$ denote the corresponding morphisms for $M \in \Nc$, then $\mu_\Nc \circ \omu_N = \mu_N$ and $\rho_\Nc \circ \orho_N = \rho_N$ for all $N \in \Nc$.
\item There exists a unique morphism $t \in \Ac(\bigoplus \M, \prod \M)$
such that $\pi_M \circ t=\rho_M$ and $t \circ \nu_M = \mu_M$ for each $M\in \M$.  
\end{enumerate}
\end{lm}
\begin{proof}

(i) It suffices to prove the existence and the uniqueness of $\rho_\N$, the second claim has a dual proof.

Consider the diagram $(M \mid  M \in \M)$ with morphisms $(\wnu_M \mid M \in \M) \in \Ac(M, \bigoplus \Nc)$ where $\wnu_M = \nu_M$ for $M\in \Nc$ and $\wnu_M=0$ otherwise. Then the claim follows from the universal property of the coproduct $(\bigoplus\M,(\nu_M \mid M\in\M))$. 

(ii) Note that for the choice $\Nc := \bigoplus ( M ) \simeq M $ we have $\onu_M = 1_M$ and the claim follows from (i).

(iii) We obtain the requested morphism by the universal property 
of the product $(\prod\M, (\pi_M  \mid M\in \M))$ applying on the cone $(\bigoplus\M,(\rho_M \mid M\in\M))$ that is provided by (ii). Dually, there exists a unique $t' \in \Ac(\bigoplus \M, \prod \M)$ with $t' \circ \nu_M = \mu_M$. Then
\begin{equation*}
\pi_M \circ (t \circ \nu_M ) = \rho_M \circ \nu_M = 1_M = \pi_M \circ \mu_M = \pi_M \circ (t' \circ \nu_M),
\end{equation*}
hence $t \circ \nu_M = \mu_M$ by the uniqueness of the associated morphism $\mu_M$ and $t = t'$ because $t'$ is the only one satisfying the condition for all $M \in \M$.
\end{proof}

We call the morphism $\rho_{\Nc}$ ($\mu_{\Nc}$) from (i) as the \emph{associated morphism} to the structural morphism $\nu_{\M}$ ($\pi_{\M}$) over the subcoproduct (the subproduct) over $\Nc$. For the special case in (ii), the morphisms $\rho_M$ ( $\mu_M$) from (ii) as the \emph{associated morphism} to the structural morphism $\nu_M$ ($\pi_M$). Let the unique morphism $t$ be called as the \emph{compatible coproduct-to-product} morphism over a family $\M$. Note that this morphism need not be a monomorphism, but it is so in case $\Ac$ being an Ab5-category \cite[Chapter 2, Corollary 8.10]{Pop73}. Moreover, $t$ is an isomorphism if the family $\M$ is finite.

\begin{lm}\label{StructAssandCtPMorphCommuteAbelianCats}
Let us use the notation from the previous lemma.
\begin{enumerate}
\renewcommand{\labelenumi}{(\roman{enumi})}  
\item For the subcoproduct over $\Nc$, the composition of the structural morphism of the subcoproduct and its associated morphism is the identity. Dually for the subproduct over $\Nc$, the composition of the associated morphism of the subproduct and its structural morphism is the identity, i.e. $\rho_\Nc \circ \nu_\Nc=1_{\bigoplus\Nc}$ and $\pi_\Nc \circ \mu_\Nc=1_{\prod\Nc}$, respectively.
\item If $\overline{t} \in \Ac\left( \bigoplus \Nc, \prod \Nc \right)$ and $t \in \Ac\left( \bigoplus \M, \prod \M \right)$ denote the compatible coproduct-to-product morphisms over $\Nc$ and $\M$ respectively, then the following diagram commutes:
\begin{equation*}
\xymatrix{ 
	\bigoplus\N  \ar@{->}[r]^{\nu_\Nc} \ar@{~>}[d]_{\overline{t}} & \bigoplus\M \ar@{.>}[r]^{\rho_\Nc}\ar@{->}[d]^{t} & \bigoplus\N  \ar@{~>}[d]^{\overline{t}} \\
	\prod\Nc \ar@{.>}[r]^{\mu_\Nc}  & \prod \M\ar@{->}[r]^{\pi_\Nc}  & \prod\N
}
\end{equation*}

\item Let $(\Nc_{\alpha} \mid \alpha<\kappa)$ be a disjoint partition of $\M$ and for $\alpha < \kappa$ let $S_\alpha:= \bigoplus \Nc_{\alpha}$, $P_{\alpha}:=\prod \Nc_{\alpha}$, and denote families of the limits and colimits like $\Sc: = (S_{\alpha} \mid \alpha < \kappa)$, $\Pc: = (P_{\alpha} \mid \alpha < \kappa)$. Then $\bigoplus\M \simeq \bigoplus{\Sc}$ and $\prod\M \simeq \prod{\Pc}$ where the both isomorphisms are canonical, i.e. for every object $M\in\M$ the diagrams commute:
\begin{equation*}
\xymatrix{
	M  \ar@{->}[r]^{\nu^{(\alpha)}_M} \ar@{->}[d]_{\nu_M} & S_{\alpha} \ar@{->}[d]^{\nu_{S_{\alpha}}} \\
	\bigoplus \M  \ar@{~>}[r]^{\simeq }  & \bigoplus \Sc  } 
	\ 
	\xymatrix{
	\prod \Pc \ar@{~>}[r]^{\simeq} \ar@{->}[d]_{\pi_{P_{\alpha}}} & \prod \M \ar@{->}[d]^{\pi_M} \\
	P_{\alpha} \ar@{->}[r]^{\pi^{(\alpha)}_M} & M}
\end{equation*}
\end{enumerate}
\end{lm}
\begin{proof} 
(i) The equality $\rho_\Nc \circ \nu_\Nc=1_{\bigoplus\Nc}$ is implied the uniqueness of the universal morphism and the equalities $(\rho_\Nc \circ \nu_\Nc) \circ \onu_N = \rho_\Nc \circ \nu_N = \onu_N$  and $1_{\bigoplus\Nc} \circ \onu_N = \onu_N$ for all $N \in \Nc$. The equality $\pi_\Nc \circ \mu_\Nc=1_{\prod\Nc}$ is dual.

(ii) We need to show that $t \circ \nu_\Nc = \mu_\Nc \circ \ot$. For all $N \in \Nc$,  $(\pi_N \circ t) \circ \nu_N=\rho_N \circ \nu_N = 1_N$ by (iii) and by (ii). But $\pi_N \circ \nu_N = 1_N$, hence $\mu_N=t \circ \nu_N$ by the uniqueness of $\mu_N$. If $\omu_N \in \Ac(N,\prod\Nc)$ denotes the unique homomorphism ensured by (ii), then the last argument proves that $\omu_N=\ot \circ \onu_N$. Thus
\begin{equation*}
\begin{split}
(t \circ \nu_\N) \circ \onu_N  &= t \circ (\nu_\N \circ \onu_N )=t \circ \nu_N=\mu_N=\mu_\N \circ \omu_N=\mu_\N \circ (\ot \circ \onu_N) = \\
&= (\mu_\N \circ \ot )\circ \onu_N
\end{split}
\end{equation*}
and the claim follows from the universal property of the coproduct $(\bigoplus\Nc,(\onu_N \mid N\in\Nc))$. The dual argument proves that $\pi_\N \circ t=\ot \circ \rho_\N$.

(iii) A straightforward consequence of the universal properties of the coproducts and products.
\end{proof}

Applying the categorical tools we have introduced we are ready to present the central notion of the paper.
Let us suppose that $\Cc$ is a subclass of objects of $\Ac$, $M$ is an object in $\Ac$ and $\Nc$ is a system of objects of $\Cc $. As the functor $\Ac(M,-)$ on any additive category maps into $\Hom$-sets with a structure of abelian groups we can define a mapping
\begin{equation*}
\Psi_\Nc: \bigoplus\left( \Ac(M,N) \mid N\in\Nc\ \right) \to \Ac(M,\bigoplus\Nc)
\end{equation*}
in the following way:

For a family of mappings $\varphi=(\varphi_N \mid N\in\Nc)$ in $\bigoplus \left( \Ac(M,N) \mid N\in\Nc \right)$ let us denote by $\Fc$ a finite subfamily such that $\varphi_N=0$ whenever $N\notin \Fc$ and let $\tau \in \Ac(M, \prod \Nc)$ be the unique morphism given by the universal property of the product $(\prod \Nc,(\pi_N \mid N\in\Fc))$ applied on the cone $(M,(\varphi_N \mid  N\in\Nc))$, i.e. $\pi_N \circ \tau=\varphi_N$ for every $N \in \Nc$. Then 
\begin{equation*}
\Psi_\Nc(\varphi)=\nu_\Fc \circ \nu^{-1} \circ \pi_\Fc \circ \tau
\end{equation*}
where $\nu \in \Ac(\bigoplus \Fc, \prod\Fc)$ denotes the isomorphism provided by Lemma~\ref{AssMorphAbelianCats}(iii).

Note that the definition $\Psi_\Nc(\varphi)$ does not depend on choice of $\Fc$ and recall an elementary observation which plays a key role in the definition of a compact object.

\begin{lm} 
The mapping $\Psi_\Nc$ is a monomorphism in the category of abelian groups for every family of objects $\Nc$.
\end{lm}
\begin{proof}  
If $\Psi_\Nc(\sigma)=0$, then $\sigma=(\rho_N \circ \sigma)_N=(0)_N$, hence $\Ker \Psi_\Nc=0$.
\end{proof}

An object $M$ is said to be \emph{$\Cc$-compact} if $\Psi_\Nc$ is an isomorphism for every family $\Nc\subseteq\Cc$,
$M$ is {\it compact} in the category $\Ac$ if it is $\Ac^{o}$-compact for the class of all object, and $M$ is \emph{self-compact} assuming $\{M\}$-compact. Note that every object is $\{0\}$-compact. 

First we formulate an elementary criterion of identifying $\Cc$-compact object.

\begin{lm}  
If $M$ is an object and a class of objects $\Cc$, then it is equivalent:
\begin{enumerate}
\item $M$ is $\Cc$-compact,
\item for every $\Nc \subseteq \Cc$ and $f \in \Ac(M, \bigoplus \Nc)$ there exists a finite subsystem $\Fc \subseteq \Nc$
and a morphism $f' \in \Ac(M, \bigoplus \Fc)$ such that  $f=\nu_\Fc \circ f'$.
\item for every  $\Nc \subseteq \Cc$ and every $f \in \Ac(M, \bigoplus \Nc)$ there exists a finite subsystem $\Fc$ contained in $\Nc$ such that  $f = \sum\limits_{F \in \Fc} \nu_F \circ \rho_F \circ f$.
\end{enumerate}
\end{lm}
\begin{proof} 
$(1) \to (2)$: Let $\Nc \subseteq \Cc$ and $f \in \Ac(M, \bigoplus \Nc)$. Then there exists a $\Psi_\Nc$-preimage $\varphi$ of $f$, hence there can be chosen a finite subsystem $\Fc \subseteq \Nc$ such that
\begin{equation*}
f=\Psi_\Nc(\varphi)=\nu_\Fc \circ \nu^{-1} \circ \pi_\Fc \circ \tau,
\end{equation*}
where we use the notation from the definition of the mapping $\Psi_\Nc$. Now it remains to put $f'=\rho_\Fc\circ f$ and utilize Lemma~\ref{AssMorphAbelianCats}(ii) to verify that
\begin{equation*}
\nu_\Fc \circ f'= \nu_\Fc \circ\rho_\Fc\circ f=\nu_\Fc \circ\rho_\Fc\circ\nu_\Fc \circ 1_{\oplus \Fc} \circ \nu^{-1} \circ \pi_\Fc \circ \tau=f.
\end{equation*}

$(2) \to (3)$: Since $\rho_\Fc \circ \nu_\Fc=1_{\bigoplus\Fc}$ by Lemma~\ref{AssMorphAbelianCats}(ii) we obtain that
\begin{equation*}
\nu_\Fc\circ\rho_\Fc\circ f=\nu_\Fc\circ\rho_\Fc\circ\nu_\Fc \circ f'=\nu_\Fc\circ f'=f.
\end{equation*}
Moreover, $\nu_\Fc\circ\rho_\Fc= \sum\limits_{F \in \Fc} \nu_F \circ \rho_F$, hence
\begin{equation*}
f= \nu_\Fc\circ\rho_\Fc\circ f = \sum\limits_{F \in \Fc} \nu_F \circ \rho_F \circ f.
\end{equation*}

$(3) \to (1)$: If we put $\varphi_F:=\rho_F \circ f$ for $F\in \Fc$ and $\varphi_N:=0$ for $N\notin \Fc$ and take $\varphi:=(\varphi_N \mid N\in\N)$, then it is easy to see that $f = \Psi_\Nc(\varphi)$ hence $\Psi_\Nc$ is surjective.
\end{proof}

Now, we can prove a characterization, which generalizes equivalent conditions well-known for the categories of modules. Note that it will play similarly important role for categorical approach to compactness as in the special case of module categories.

\begin{thm}\label{e.c.}  
	The following conditions are equivalent for an object $M$ and a class of objects $\Cc$:
\begin{enumerate}
\item $M$ is not $\Cc$-compact,
\item there exists a countably infinite system $\Nc_\omega$  
of objects from $\Cc$ and $\varphi \in \Ac(M,\bigoplus\Nc_\omega)$ such that $\rho_N \circ \varphi\ne 0$ for every $N\in\Nc_\omega$,
\item for every system $\Gc$ of $\Cc$-compact objects and every
epimorphism $e \in \Ac(\bigoplus\Gc, M)$ there exists a countable subsystem $\Gc_\omega \subseteq \Gc$ such that $f^c \circ e \circ \nu_{\Gc_\omega}\ne 0$ for the cokernel $f^c$ of every morphism $f\in\Ac(F,M)$ where $F$ is a $\Cc$-compact object. 
\end{enumerate}
\end{thm}
\begin{proof} 
$(1) \to (2)$: Let $\Nc$ be a system of objects from $\Cc$ for which there exists a morphism $\varphi\in\Ac(M,\bigoplus \Nc)\setminus \im\Psi_\Nc$. Then it is enough to take $\N_\omega$ as any countable subsystem of the infinite system $(N\in\Nc \mid \rho_N \circ \varphi\ne 0)$.

$(2) \to (3)$ Let $\Gc$ be a family of  $\Cc$-compact objects and  $e\in\Ac(\bigoplus\Gc, M)$ an epimorphism. If $N\in \Nc_\omega$, then $(\rho_N \circ \varphi) \circ e\ne 0$, hence by the universal property of the coproduct $\bigoplus\Gc$ applied on the cone $(N, (\rho_N \circ \varphi \circ e \circ \nu_{G} \mid G\in\Gc))$ there exists $G_N\in\Gc$ such that $\Ac(G_N, N) \ni \rho_N \circ \varphi \circ e \circ \nu_{G_N}\ne 0$. Put $\Gc_\omega=\left( G_N \mid N\in\Nc_\omega \right)$, where every object from the system $\Gc$ is taken at most once, i.e. we have a canonical monomorphism $\nu_{\Gc_{\omega}} \in \Ac \left(\bigoplus \Gc_{\omega}, \bigoplus \Gc \right)$.

Assume to the contrary that there exist a $\Cc$-compact object $F$ and a morphism $f \in \Ac(F,M)$ such that $f^c \circ e \circ \nu_{\Gc_\omega}= 0$ where $f^c \in \Ac(M, \Coker(f))$ is the cokernel of the morphism $f$.
Let $N\in\Nc_\omega$ and, furthermore, assume that $\rho_N \circ \varphi \circ f=0$. Then the universal property of the cokernel ensures the existence of a morphism $\alpha\in\Ac(\Coker(f),N)$ such that $\alpha \circ f^c=\rho_N\circ \varphi$, i.e. that commutes the diagram:
\begin{equation*}
\xymatrix{
	& & F \ar@{->}[d]^{f} & \\
	\bigoplus\Gc_{\omega}  \ar@{->}[r]^{\nu_{\Gc_{\omega}}} & \bigoplus \Gc \ar@{->}[r]^{e} & M \ar@{->}[r]^{\varphi} \ar@{->}[d]^{f^{c}} & \bigoplus \Nc_{\omega} \ar@{->}[d]^{\rho_N} \\
	& & \Coker(f) \ar@{~>}[r]^{\alpha} & N }
\end{equation*}

Thus $(\rho_N \circ \varphi ) \circ e \circ \nu_{\Gc_\omega}= (\alpha \circ f^c ) \circ e \circ \nu_{\Gc_\omega}=0$, which contradicts to the construction of $\Gc_\omega$. We have proved that $\rho_N \circ ( \varphi \circ f) \ne 0$ for each $N\in\Nc_\omega$, hence $\varphi \circ f\in \Ac(F, \bigoplus \Nc) \setminus \im \Psi_{\Nc_{\omega}}$. We get the contradiction with the assumption that $F$ is $\Cc$-compact, thus $f^c \circ e \circ \nu_{\Gc_\omega}=\ne 0$ 

$(3) \to (1)$: If $M$ is $\Cc$-compact itself, then the system $\Gc=(M)$ and the identity map $e$ on $M$ are counterexamples for the condition $(3)$.
\end{proof}

\begin{cor}
If $\Ac$ contains injective envelopes $E(U)$ for all objects $U\in\Cc$, then an object $M$ is not compact if and only if there exists a (countable) system of injective envelopes $\Ec$ in $\Ac$ of objects of $\Cc$ for which $\Psi_\Nc$ is not surjective for some subsystem $\Nc$ of $\Cc$.
\end{cor}
\begin{proof}
By the previous proposition, it suffices to consider the composition of $\varphi \in \Ac(M, \bigoplus \Nc_\omega) \setminus \im\Psi_{\Nc_\omega}$ where $\Nc_\omega$ witnesses that $M$ is not $\Cc$-compact and the canonical morphism $\iota \in \Ac \left( \bigoplus \Nc_{\omega}, \bigoplus \Ec \right)$, where we put $\Ec := \left(E(N) \mid N\in\Nc_{\omega} \right)$. 
\end{proof}	

\section{Classes of compact objects }

Let us denote by $\Ac$ a complete abelian category and $\Cc$ a class of some objects of $\Ac$.
First, notice that closure properties of the class of $\Cc$-compact objects are similar to closure properties of classes of dually slender modules since their follows by the fact that the contravariant functor $\Ac(-,\bigoplus\Nc)$ commutes with finite coproducts and it is left exact. We present a detailed proof of the fact that the class of all $\Cc$-compact objects is closed under finite coproducts and cokernels using Theorem~\ref{e.c.}.

\begin{lm}  
	The class of all $\Cc$-compact objects is closed under finite direct sums and all cokernels of morphisms $\alpha \in \Ac(M, C)$ where $C$ is $\Cc$-compact and $M$ is arbitrary.
\end{lm}
\begin{proof} 
Suppose that $\bigoplus_{i=1}^n M_i$ is not $\Cc$-compact.
Then by Theorem~\ref{e.c.} there exist a sequence $(N_i \mid i<\omega)$ of objects and a morphism $\varphi \in \Ac (\bigoplus_{i=1}^nM_i, \bigoplus_{j<\omega}N_j)$ such that $\rho_j \circ \varphi \ne 0$ for each $j<\omega$. Since 
$\omega= \bigcup_{i=1}^n\{j<\omega \mid \rho_j \circ \varphi \circ \nu_i \ne 0\}$ there exists $i$
for which the set $\{j<\omega\mid \rho_j \circ \varphi \circ \nu_i \ne 0\}$ is infinite, hence
$M_i$ is not $\Cc$-compact by applying Theorem~\ref{e.c.}.

Similarly, suppose that $\alpha^{c}$ is the cokernel of $\alpha \in \Ac(M, C)$, where $\Coker(\alpha) $ is not $\Cc$-compact, and $\varphi \in \Ac(\Coker(\alpha), \bigoplus_{j<\omega}N_j)$ for $(N_i \mid i<\omega)$ satisfies $\rho_j \circ \varphi\ne 0$ for every $j<\omega$. Then, obviously, $\rho_j \circ \varphi \circ \pi \ne 0$ for each $j<\omega$ and so $C$ is not $\Cc$-compact
again by Theorem~\ref{e.c.}.
\end{proof}

\begin{lm} 
If $\M$ is an infinite system of objects in $\Ac$ satisfying that for each $M\in\M$ there exists $C\in\Cc$ such that $\Ac(M,C)\ne 0$, then $\bigoplus \M$ is not $\Cc$-compact.
\end{lm}
\begin{proof}
 It is enough to take $\Nc=( C_M \mid  M\in \M)$ where $\Ac(M,C_M)\ne 0$ and apply Theorem~\ref{e.c.}(ii).
\end{proof}

We obtain the following consequence immediately:

\begin{cor} 
Let $\M$ be a system of objects of $\Ac$. Then $\bigoplus \M$ is $\Cc$-compact if and only if all $M\in\M$ are $\Cc$-compact and there exists $C_M\in\Cc$ such that $\Ac(M,C_M)\ne 0$ for only finitely many $M\in\M$.
\end{cor}

Let us confirm that relativized compactness behaves well under taking finite unions of classes and verify with an example that this closure property can not be extended to an infinite case. 

\begin{lm} 
	Let $\Cc_1,\dots,\Cc_n$ be a finite number of classes of objects and let $C \in \Ac$. Then $C$ is $\bigcup_{i=1}^n\Cc_i$-compact if and only it $C$ is $\Cc_i$-compact for every $i\le n$.
\end{lm}
\begin{proof} 
	The direct implication is trivial. If $C$ is not $\bigcup_{i=1}^n\Cc_i$-compact, there exists a sequence $(C_i \mid i<\omega)$ of objects of $\bigcup_{i=1}^n\Cc_i$ with a morphism $\varphi \in \Ac(C, \bigoplus_{j<\omega}B_j)$ such that $\rho_j \circ \varphi \ne 0$ for every $j<\omega$ by Theorem~\ref{e.c.}. Since there exists $k\le n$ for which infinitely many $C_i$'s belong to $\Cc_j$ we can see that $C$ is not $\Cc_j$-compact by Theorem~\ref{e.c.}.
\end{proof}

\begin{exm}\label{e0} 
Let $R$ be a ring over which there is an infinite set of non-isomorphic simple right modules. Any non-artinian Von Neumann regular ring serves as an example where the property holds. Suppose that $\Ac$ is the full subcategory of category consisting of all semisimple right modules, which is generated by all simple modules. Fix a countable sequence $S_i$, $i<\omega$, of pair-wisely non-isomorphic simple modules. Then the module $\bigoplus_{i<\omega}S_i$ is $\{S_i\}$-compact for each $i$ but it is not $\bigcup_{i < \omega} \{S_i\}$-compact.
\end{exm}

Relative compactness of an object is preserved if we close the class under all cogenerated objects.

\begin{lm} 
Let $Cog(\Cc)$ be the class of all objects cogenerated by $\Cc$. Then every $\Cc$-compact object is $Cog(\Cc)$-compact.
\end{lm}
\begin{proof} 
Let us suppose that an object $C$ is not $Cog(\Cc)$-compact and fix a sequence $\Bc:= (B_i \mid i<\omega)$ of objects of $Cog(\Cc)$ and a morphism $\varphi \in \Ac (C, \bigoplus \Bc)$ such that $\rho_j \circ \varphi\ne 0$ for each $j<\omega$ which exists by Theorem~\ref{e.c.}. Since $Cog(\Cc)$ is closed under subobjects we may suppose that $\rho_j \circ \varphi$ are epimorphisms. Furthermore, for every $j<\omega$ there exists a non-zero morphism $\tau_j \in \Ac(B_j, T_j)$ with $T_j\in\Cc$. Form the sequence $\Tc:=(T_i \mid i<\omega)$.  Let $\tau$ be the uniquely defined morphism from $\Ac(\bigoplus \Bc, \bigoplus \Tc)$ satisfying $\tau \circ \nu_j = \overline{\nu}_j \circ \tau_j$. Then $\overline{\rho}_j\circ \tau \circ \nu_i=\overline{\rho}_j \circ \overline{\nu}_i \circ \tau_i$ which is equal to $\tau_i$  whenever $i=j$ and it is zero otherwise, hence $\overline{\rho}_i\circ \tau \circ \nu_i \circ \rho_i=\overline{\rho}_i \circ \tau$ by the universal property of $\bigoplus \Bc$. Finally, since $\rho_i \circ \varphi$ is an epimorphism and $\tau_i$ is non-zero $\tau_i \circ \rho_i \circ \varphi \ne 0$ and so
\begin{equation*}
\overline{\rho}_j \circ \tau \circ \varphi = \overline{\rho}_i \circ \tau \circ \nu_i \circ \rho_i \circ \varphi = \overline{\rho}_i \circ \overline{\nu}_i \circ \tau_i \circ \rho_i \circ \varphi = \tau_i \circ \rho_i \circ \varphi \ne 0
\end{equation*}
for every $i<\omega$. Thus the composition $\tau \circ \varphi$ witnesses that $C$ is not $\Cc$-compact again by Theorem~\ref{e.c.}.
\end{proof}

A complete abelian category $\Ac$ is called \emph{$\Cc$-steady}, if there exists an $\Ac$-projective $\Cc$-compact object $G$ which finitely generates the class of all $\Cc$-compact objects, i.e. for every $\Cc$-compact object $F$ there exists $n\in \mathbb N$ and a homomorphism $h \in \Ac(G^{(n)}, F)$ such that cokernel of $h$ is isomorphic to $F$. $\Ac$ is said to be \emph{steady} whenever it is an $\Ac^{o}$-steady category for the class $\Ac^{o}$ of all objects of $\Ac$.

\begin{exm}
If $R$ is a right steady ring and $\Ac=\Mod$-$R$ is the category of all right $R$-modules, then $\Cc$-compact objects are precisely dually slender modules and $\Ac$ is a steady category.
\end{exm}

Furthermore, in \cite[Theorem 1.7]{GomMilNas94} it was proved that a locally noetherian Grothendieck category is steady.

Recall that an object $A$ is \emph{simple} if for every $B$ and any non-zero morphism from $\Ac(A,B)$ is a monomorphism and an object is \emph{semisimple} if it is isomorphic do direct sum of simple objects.
We characterize steadiness of categories such that all their object are semisimple.

\begin{lm} 
	Let $\Ac$ be generated by a class of simple objects $\mathcal S$. Then $\Ac$ is $\Cc$-steady if and only if there are only finitely many non-isomorphic simple objects $S$ such that $\Hom(S,C)\ne 0$ for some object $C\in \Cc$. 
\end{lm}
\begin{proof} 
	If there exist infinitely many non-isomorphic simple objects $S$ such that $\Ac(S, T_S)\ne 0$ for suitable $T_S \in \Cc$, it is enough to take countably infinite family of such simple objects $\Sc := (S_i \mid i<\omega)$ and non-zero morphisms $\tau_i \in \Ac(S_i, T_{S_i})$ where $T_{S_i}\in \Cc$ are collected into $\Tc := (T_{S_i} \mid i<\omega)$.  Then the morphism $\tau \in \Ac(\bigoplus \Sc, \bigoplus \Tc)$ defined by the universal property of a direct sum witnesses that $\bigoplus \Sc$ is not $\Cc$-compact. Since it is easy to see that for every epimorphism from $\Ac(A, \bigoplus \Sc)$, the object $A$ is not $\Cc$-compact as well, yielding that the category $\Ac$ is not steady.
\end{proof}

\begin{exm} 
	Let $\Ac$ be a category semisimple right modules over ring with infinite set of non-isomorphic simple right modules as in Example~\ref{e0}. Then $\Ac$ is not steady, and if the ring $R$ is right steady, which is true for example for each countable commutative regular ring,
the category of all right $R$-modules $\Mod$-$R$ steady.
\end{exm}

We say that a complete abelian category $\Ac$ is $\prod\Cc$-compactly generated if there is a set $\Gc$ of objects of $\Ac$ that generates $\Ac$ and the product of any system of objects in $\Gc$ is $\Cc$-compact. Note that $\Gc$ consists only of $\Cc$-compact objects. 

\begin{lm} 
If $E$ is a  $\Cc$-compact injective generator of $\Ac$ such that there exists a monomorphism  $m \in \Ac(E^{(\omega)}, E)$, then $\Ac$ is $\prod\Cc$-compactly generated. 
\end{lm}
\begin{proof}
It follows immediately from Theorem~\ref{e.c.}(iii).
\end{proof}

\begin{exm}
Let $R$ be a right self-injective, purely infinite ring. Then $E:=R$ is an injective generator and there is an embedding $0 \to R^{(\omega)} \to R$. By the previous lemma, the category $\Mod$-$R$ is $\prod \Cc$-generated.
\end{exm}

\section{Products of compact objects}

We start the section by an observation that the cokernel of the compatible coproduct-to-product morphism over a countable family is $\Cc$-compact where $\Cc$ is a class of objects in an abelian category $\Ac$. This initial step will be later extended to families regardless of their cardinality.

\begin{lm}\label{CokCtPMorphCountIsComp}
	Let $\Ac$ be $\prod\Cc$-compactly generated and let $\M$ be a countable family of objects in $\Ac$. If $t \in \Ac \left( \bigoplus\M, \prod\M \right)$ is the compatible coproduct-to-product morphism, then $\Coker(\nu)$ is $\Cc$-compact.
\end{lm}

\begin{proof} 
As for a finite $\M$ there is nothing to prove, suppose that $\M = (M_n\mid n<\omega)$. 
Let $\Gc$ be a family of objects of $\Ac$ such that every product of a system of objects in $\Gc$ is $\Cc$-compact and let $e\in\Ac(\bigoplus \Gc,\prod\M)$ be an epimorphism, which exists by the hypothesis.
Let $t^c$ be the cokernel of $t$. Then both $t^c$ and $e':= t^c \circ e$ are epimorphisms and $t^c \circ t=0$. We will show that for every countable subsystem $\Gc_\omega$ of $\Gc$ there exists a $\Cc$-compact object $F$ and a morphism $f\in\Ac(F,\Coker(t))$ such that $\Ac \left( \bigoplus \Gc_{\omega}, \Coker(f) \right)  \ni f^c \circ e' \circ \nu_{\Gc_\omega} = 0$ for the cokernel $f^c \in \Ac(\Coker(t), \Coker(f))$. By Theorem~\ref{e.c.} this yields that $\Coker(t)$ is $\Cc$-compact.

Since for any finite $\Gc_\omega \subseteq \Gc$ it is enough to take $F:=\bigoplus \Gc_\omega$ and $f:=e' \circ \nu_{\Gc_\omega}$, we may fix a countably infinite family $\Gc_\omega= (G_n \mid n<\omega) \subseteq \Gc$. For each $n<\omega$ put $\Gc_n=(G_i \mid i\le n)$ and let $\pi_{\Gc_n} \in \Ac \left( \prod\Gc_\omega, \prod\Gc_n \right)$ and $\pi_{M_n} \in \Ac \left( \prod\M, M_n \right)$ denote the structural morphisms, and let $\ou^{-1} \in \Ac \left(\prod\Gc_n, \bigoplus\Gc_n \right)$ be the inverse of the compatible coproduct-to-product morphism $\ou \in \Ac \left(\bigoplus\Gc_n, \prod\Gc_n\right)$ that exists for finite families.

First, let us fix $n \in \omega$ and we prove that $\nu_{G_k} = \nu_{\Gc_n} \circ \ou^{-1} \circ \pi_{\Gc_n} \circ \mu_{G_k}$.  Let $\onu_{G_k} \in \Ac(G_k, \bigoplus\Gc_n)$ be the structural morphism of the coproduct $\bigoplus\Gc_n$ and let  $\omu_{G_k}\in\Ac(G_k,\prod\Gc_n)$ be the associated morphism to the product $\prod\Gc_n$. Since $\nu_{\Gc_n} \circ \onu_{G_k}=\nu_{G_k}$ and  $\mu_{G_k} = u \circ \nu_{G_k}$  (by Lemma~\ref{AssMorphAbelianCats}(iii)), then we immediately infer the following equalities from Lemma~\ref{StructAssandCtPMorphCommuteAbelianCats}(ii) :
\begin{equation*}
\begin{split}
\nu_{G_k} &= \nu_{\Gc_n} \circ \onu_{G_k} = \nu_{\Gc_n} \circ (\ou^{-1} \circ \ou \circ \onu_{G_k}) = (\nu_{\Gc_n} \circ \ou^{-1} ) \circ \ou \circ \onu_{G_k} = \\
 &= (\nu_{\Gc_n} \circ \ou^{-1} ) \circ (\pi_{\Gc_n} \circ u \circ \nu_{\Gc_n}) \circ \onu_{G_k} = (\nu_{\Gc_n} \circ \ou^{-1} ) \circ \pi_{\Gc_n} \circ u \circ \nu_{G_k} = \\
 &= (\nu_{\Gc_n} \circ \ou^{-1} ) \circ \pi_{\Gc_n} \circ \mu_{G_k}
\end{split}
\end{equation*}

Now, if we employ the universal property of the product $(\prod \M, (\pi_{M_n} \mid n <\omega))$ with respect to the cone $(\prod \Gc_\omega, (\pi_{M_n} \circ e \circ \nu_{\Gc_n} \circ \ou^{-1} \circ \pi_{\Gc_n}  \mid n <\omega))$, then there exists a unique morphism $\alpha\in\Ac(\prod \Gc_\omega,\prod \M)$ such that the middle non-convex pentagon in the following diagram commutes :
\begin{equation*}
\xymatrix{
	G_k  \ar@{->}@/^2pc/[rr]^{\mu_{G_k}} \ar@{->}[r]^{ \nu_{G_k}} \ar@{=}[d]_{1_{G_k}} & \bigoplus \Gc_{\omega}   \ar@{->}[r]^{u} & \prod\Gc_{\omega} \ar@{->}[d]^{\pi_{\Gc_n}} \ar@{->}[dr]^{\pi_{M_n} \circ \alpha}  \ar@{~>}@/^1pc/[drr]^{\alpha} & & \bigoplus \M \ar@{->}[d]^{t}   & \\
	G_k \ar@{->}[r]_{ \onu_{G_k}} \ar@{=}[d]_{1_{G_k}} & \bigoplus \Gc_{n} \ar@{->}[u]_{\nu_{\Gc_n}} \ar@{->}[r]^{\ou}_{\simeq} \ar@{->}[d]^{\nu_{\Gc_{\omega}} \circ \nu_{\Gc_n} } & \prod \Gc_{n}  & M_n  &\prod \M \ar@{->}[d]^{t^c} \ar@{->}[l]_{\pi_{M_n}}&  \\
    G_k \ar@{->}[r]_{ \wnu_{G_k}} & \bigoplus \Gc \ar@{->}[urr]^{\pi_{M_n} \circ e} \ar@{->}@/_1pc/[urrr]^{e} & & &\Coker(t) \\ 
}
\end{equation*}
Then for each $k\le n$ we deduce that
\begin{equation*}
\begin{split}
\pi_{M_n} \circ (\alpha \circ \mu_{G_k} - e \circ \wnu_{G_k}) &= \pi_{M_n} \circ (\alpha \circ \mu_{G_k} - e \circ \nu_{\Gc_{\omega}} \circ \nu_{\Gc_n} \circ \ou^{-1} \circ \pi_{\Gc_n} \circ \mu_{G_k}) = \\
& = (\pi_{M_n} \circ  \alpha - \pi_{M_n} \circ e \circ \nu_{\Gc_{\omega}} \circ\nu_{\Gc_n} \circ \ou^{-1} \circ \pi_{\Gc_n})\circ \mu_{G_k} = 0
\end{split} 
\end{equation*}
then $\alpha \circ \mu_{G_n} = e \circ \wnu_{G_n}$ for every $n <\omega$ is yielded as the number $n$ was fixed. Note that $\prod \Gc_\omega$ is $\Cc$-compact by the hypothesis. Now, consider $f^c$ the cokernel of the morphism $f = t^c \circ \alpha \in \Ac \left( \prod \Gc_\omega, \Coker(t)\right)$.
Then 
\begin{equation*}
\begin{split}
0 &= f^c \circ t^c \circ (e \circ \wnu_{G_k}- \alpha \circ \mu_{G_n})= \\ 
&= f^c \circ t^c \circ e \circ \wnu_{G_n} - f^c \circ t^c \circ \alpha \circ \mu_{G_n} = f^c \circ e' \circ \wnu_{G_n}
\end{split}
\end{equation*}
hence $0 = f^c \circ e' \circ \wnu_{\Gc_n} = f^c \circ e' \circ \nu_{\Gc_\omega} \circ \nu_{\Gc_k}$ for every $n<\omega$, which finishes the proof.
\end{proof}

Let $\Ic\ne\emptyset$ be a system of subsets of a set $X$. We recall that $\Ic$ is said to be 
\begin{enumerate}
\item[--]\emph{an ideal} if it is closed under subsets (i.e. if $A \in \Ic$ and $B \subseteq A$, then $B \in \Ic$) and 
under finite unions, (i.e. if $A, B\in \Ic$ , then $A \cup B \in \Ic$),
\item[--] \emph{a prime ideal} if it is a proper ideal and for all subsets $A$, $B$ of $X$, $A \cap B \in \Ic$ implies that $A \in \Ic$ or $B \in \Ic$,
\item[--] \emph{a principal ideal} if there exists a set $Y\subseteq X$ such that $\Ic=P(Y)$.
\end{enumerate}
The set $\Ic \mid Y = \{ Y \cap A \mid A \in \Ic \}$ is called \emph{a trace} of on ideal $\Ic$ on $Y$.

Note that the trace of an ideal is also an ideal and that $\Ic$ is a prime ideal if and only if for every $A \subseteq X$, $A \in \Ic$ or $X \setminus A \in \Ic$. 
Moreover, a principal prime ideal on $X$ is of the form $P(X\setminus\{x\})$ for some $x\in X$.

Dually, a system $\mathcal F\ne\emptyset$ of non-empty subsets of $X$ is said to be
\begin{enumerate}
\item[--] \emph{a filter} if it is closed under finite intersections and supersets,
\item[--] \emph{an ultrafilter} if it is a filter which is not properly contained in any other filter on $X$,
\end{enumerate}
We say that a filter $\mathcal F$ is {\it $\lambda$-complete}, if $\bigcap\mathcal G\in \mathcal F$ for every subsystem $\mathcal G\subseteq\mathcal F$ such that $|\mathcal G|<\lambda$
and $\mathcal F$ is {\it countably complete}, if it is $\omega_1$-complete.

Note that there is a one-to-one correspondence between ultrafilters and prime ideals on $X$ defined by $\Ic \mapsto \Pc(X)\setminus \Ic$ for an ideal $\Ic$.

Now, we are able to generalize \cite[Lemma 3.3]{KalZem2014}

\begin{prop}\label{PrimeIdealofAnnihSubfamofNonCompProd}  
Let $\Ac$ be a $\prod\Cc$-compactly generated category,
$\M$ a family of $\Cc$-compact objects of $\Ac$ and $\Nc=(N_n \mid n<\omega)$ a countable family of objects of $\Cc$. Suppose that $\Psi_\Nc$ is not surjective and fix $\varphi\in\Ac(\prod \M,\bigoplus\Nc)\setminus\Img\Psi_\Nc$. If we denote 
$\Ic_n = \{ \J \subseteq \M \mid \rho_{N_k} \circ \varphi \circ\mu_\J=0 \ \forall k\ge n\}$ and $\Ic = \bigcup_{n < \omega} \Ic_n \subseteq \Pc(\M)$, then the following holds:
\begin{enumerate}
\renewcommand{\labelenumi}{(\roman{enumi})}  
  \item $\Ic_n$ is an ideal for each $n$,
  \item $\Ic$ is closed under countable unions of subfamilies,
  \item there exists $n < \omega$ for which $\Ic = \Ic_n$,
  \item there exists a subfamily $\Uc \subseteq \M$ such that the trace of $\Ic$ on $\Uc$ forms a non-principal prime ideal.
\end{enumerate}
\end{prop}
\begin{proof} 
Let $\Gc$ be a set of $\Cc$-compact objects satisfying that every product of a system of objects in $\Gc$ is $\Cc$-compact, which is guaranteed by the hypothesis.

(i) Obviously, $\emptyset\in \Ic_n$ and $\Ic_n$ is closed under subsets.
The closure of $\Ic_n$ under finite unions follows from Lemma~\ref{StructAssandCtPMorphCommuteAbelianCats}(iii) applied on the disjoint decomposition $\J\cup \K=\J\cup (\K\setminus \J)$,
i.e. from the canonical isomorphism $\prod\J\cup \K \cong \prod\J\times \prod\K \setminus \J$.

(ii) First we show that $\Ic$ is closed under countable unions of pairwise disjoint sets. Let $\K_j, j<\omega$ be pair-wisely disjoint subfamilies of $\Ic$ and put $\K=\bigcupdot_{j < \omega} \K_j$. Let $K_i: =\prod \K_i$. We show that there exists $k < \omega$ such that $\K_j \in \Ic_k$ for each $j < \omega$. Assume that for all $n < \omega$ there exist possibly distinct $i(n)$ such that $\K_{i(n)} \notin \Ic_n$. Hence $\rho_{N_l(n)} \circ \varphi \circ \mu_{\K_{i(n)}} \ne 0$ for some $l(n) \geq n$ and there is a $\Cc$-compact generator
$G_n\in\Gc$ and a morphism $f_n\in\Ac(G_n, K_{i(n)})$ with $\rho_{N_{l(n)}} \circ \varphi \circ \mu_{\K_{i(n)}} \circ f_n \ne 0$. Set $\K':= ( K_{i(n)} \mid n < \omega)$, $\wK:= ( K_{n} \mid n < \omega)$. 

Put $\Gc_{\omega} := ( G_j \mid j < \omega)$ and denote by $(\prod\Gc_{\omega},(\pi_{G_j} \mid j < \omega))$ the product of $\Gc_{\omega}$ and by $\mu_{G_j} \in \Ac\left( G_j, \prod\Gc_{\omega} \right)$, $j<\omega$, the associated morphisms given by Lemma~\ref{AssMorphAbelianCats}(i). Then the universal property of the product $\prod \K'$ applied to the constructed cone gives us a morphism $f \in \Ac \left(\prod \Gc_{\omega}, \prod \K' \right)$ such that $f_n \circ \pi_{G_n}=\pi_{K_{i(n)}} \circ f$, hence
\begin{equation*}
f_n = f_n \circ \pi_{G_n} \circ \mu_{G_n} = \pi_{K_{i(n)}} \circ f \circ \mu_{G_n} = \pi_{K_{i(n)}} \circ \mu_{K_{i(n)}} \circ f_n 
\end{equation*}
Since $\prod\Gc_{\omega}$ is $\Cc$-compact by the hypothesis there exists arbitrarily large $m < \omega$ such that $\rho_{N_{l(m)}} \circ \varphi \circ \mu_{\K'} \circ f=0$ where $\mu_{\K'} \in \Ac \left (\prod\K', \prod\M \right)$ is the associated morphism to $\pi_{\K'} \in \Ac(\prod\M, \prod\K')$ over the subcoproduct of $\K'$. Hence 
\begin{equation*}
\rho_{N_{l(m)}} \circ \varphi \circ (\mu_{K_{i(m)}} \circ f_m)=\rho_{N_{l(m)}} \circ \varphi\circ \mu_{\K'} \circ f \circ \mu_{G_m} = 0,
\end{equation*}
a contradiction.

We have proved that there is some $n< \omega$ such that $\rho_  {N_k} \circ \varphi \circ \mu_{\K_j}=0$ for each $k\ge n$ and $j<\omega$, without loss of generality we may suppose that $n=0$. Denote by $t^c$ the cokernel of the compatible coproduct-to-product morphism $t \in \Ac \left( \bigoplus\K, \prod\K \right)$. 
As $\varphi \circ \mu_{\K} \circ t=0$, the universal property of the cokernel ensures the existence of the morphism $\tau\in\Ac(\Coker(t),\bigoplus\N)$ such that $\varphi \circ \mu_\K=\tau \circ t^c $. Hence there exists $n< \omega$ such that $\rho_{N_k} \circ \varphi \circ \mu_{\K}=0$ for each $k\ge n$ since $\Coker(t)$ is $\Cc$-compact by Lemma~\ref{CokCtPMorphCountIsComp}, which proves that $\K\subseteq\Ic_n$.

To prove the claim for whatever system $(\J_j \mid j < \omega)$ in $\Ic$ is chosen, it remains to put $\J_0 = \K_0$ and $\J_i = \K_i \setminus \bigcup_{j < i} \K_j$ for $i > 0$.

(iii) Assume that $\Ic \ne \Ic_j$ for every $j < \omega$. Then there exists a countable sequence of families of objects $(\J_j \in \Ic \setminus \Ic_j \mid j \in \omega)$. By (ii) we get $\J:= \bigcup_{j<\omega} \J_j \in \Ic$ and there is some $n < \omega$ such that $\J \in \Ic_n$. Having $\J_n \subseteq \J \in \Ic_n$ leads us to a contradiction. 

(iv) We will show that there exists a family $\Uc \subseteq \M$ such that for every $\K \subseteq \Uc, \K \in \Ic$ or $\Uc \setminus \K \in \Ic$. Assume that such $\Uc$ does not exist. Then we may construct a countably infinite sequence of disjoint families $(\K_i \mid i < \omega)$ where $\K_i$ are non-empty for $i>0$ in the following way: Put $\K_0 = \emptyset$ and $\J_0=\M$. There exist disjoint sets $\J_{i+1}, \K_{i+1} \subset \J_i$ such that $\J_i=\J_{i+1}\cup \K_{i+1}$ where $\J_{i+1}, \K_{i+1}\not \in \Ic$. Now, for each $n \geq 1$ there exists a compact generator $G_n\in\Gc$ and a morphism $f_n\in\Ac(G_n,\prod\K_n)$ such that $\rho_{N_k} \circ \varphi \circ \mu_{\K_n} \circ f_n \ne 0$ for some $k>n$ which contradicts to the fact that $\prod_{n < \omega} G_n$ is $\Cc$-compact (hence $\rho_{N_k} \circ \varphi \circ \mu_{\K_n} \circ f_n \circ \pi_n = 0$ starting from some large enough $k < \omega$).

The trace of $\Ic$ on $\Uc$ is a prime ideal and assume that it  is principal, i.e. it consists of all subfamilies of $\Uc$ excluding one particular index $U \in \Uc$, so $\Ic \mid \Uc = \Pc(\Uc \backslash \{ U\})  \in \Ic$. On the other hand, $U$ is $\Cc$-compact itself, which implies $\{U\} \in \Ic$. This yields $\Ic \mid \Uc$ containing $\Uc$, a contradiction. 
\end{proof}

As a consequence of Proposition \ref{PrimeIdealofAnnihSubfamofNonCompProd} we can formulate
a generalization of \cite[Theorem 3.4]{KalZem2014}:

\begin{cor}\label{CompGenCatsClosedUnderCountProds}
Let $\Ac$ be a $\prod\Cc$-compactly generated category.
Then the following holds:
\begin{enumerate}
\renewcommand{\labelenumi}{(\roman{enumi})}  
  \item A product of countably many $\Cc$-compact objects is $\Cc$-compact.
  \item If there exists a system $\M$ of cardinality $\kappa$ of $\Cc$-compact objects such that the product $\prod \M$ is not $\Cc$-compact, then there exists an uncountable cardinal $\lambda < \kappa$ and a countable complete nonprincipal ultrafilter on $\lambda$.
\end{enumerate}
\end{cor}
\begin{proof} 
(i) An immediate consequence of Proposition \ref{PrimeIdealofAnnihSubfamofNonCompProd}(iii).

(ii) Let $\M$ be a system of cardinality $\kappa$ of $\Cc$-compact objects and suppose that $\prod \M$ is not a $\Cc$-compact object. Then there exists a countable family $\Nc$ such that $\Psi_{\Nc}$ is not surjective. By  Lemma~\ref{PrimeIdealofAnnihSubfamofNonCompProd}(iv) there exists a subfamily $\Uc \subseteq \M$ such that the trace of $\Ic$ on $\Uc$ forms a non-principal prime ideal which is closed under countable unions of families by Lemma~\ref{PrimeIdealofAnnihSubfamofNonCompProd}(ii). If we define $\Vc = \Pc(\Uc)  \setminus (\Ic \mid \Uc )$ then $\Vc$ forms a countable complete non-principal ultrafilter on $\Uc$. It is uncountable by applying (i).
\end{proof}

Before we formulate the main result of this section which answers the question from \cite{ElbKepNem2003} in Abelian categories, let us list several set-theoretical notions and their properties guaranteeing that the hypothesis of the theorem is consistent with ZFC.

A cardinal number $\lambda$ is said to be {\it measurable} if there exists a $\lambda$-complete non-principal ultrafilter on $\lambda$ and it is {\it Ulam-measurable} if there exists a countably complete non-principal ultrafilter on $\lambda$. A regular cardinal $\kappa$ is {\it strongly inaccessible} if $2^\lambda<\kappa$ for each $\lambda<\kappa$.
Recall that

	\begin{enumerate}
		\item[$\bullet$]\cite[Theorem 2.43.]{Zel2011} every Ulam-measurable cardinal is greater or equal to the first measurable cardinal;
		\item[$\bullet$]\cite[Theorem 2.44.]{Zel2011} every measurable cardinal is strongly inaccessible;
		\item[$\bullet$]\cite[Corollary IV.6.9]{Kun80} it is consistent with ZFC that there is no strongly inaccessible cardinal.
	\end{enumerate}

\begin{thm}\label{CompGenCatsClosedUnderProds}
Let $\Ac$ be a $\prod\Cc$-compactly generated category,
$\M$ a family of $\Cc$-compact objects of $\Ac$.
If we assume that there is no strongly inaccessible cardinal, then every product of $\Cc$-compact objects is $\Cc$-compact.
\end{thm}
\begin{proof}
Suppose that the product of an uncountable system of $\Cc$-compact objects is not $\Cc$-compact. Then Corollary~\ref{CompGenCatsClosedUnderCountProds}(ii) ensures the existence of a countable complete ultrafilter on $\lambda$. Thus there exists a measurable cardinal $\mu\le\lambda$, which is necessarily strongly inaccessible.
\end{proof}

\def\bibname{Bibliography}

\end{document}